\documentclass[12pt]{amsart}

\newtheorem{theorem}{Theorem}[section]
\newtheorem{lemma}[theorem]{Lemma}

\newcommand{\calE}{\mathcal{E}}
\newcommand{\calF}{\mathcal{F}}

\numberwithin{equation}{section}

\begin{document}

\baselineskip=15pt

\title{Einstein--Hermitian connection on twisted Higgs bundles}

\author[I. Biswas]{Indranil Biswas}

\address{School of Mathematics, Tata Institute of Fundamental
Research, Homi Bhabha Road, Bombay 400005, India}

\email{indranil@math.tifr.res.in}

\author[T. G\'omez]{Tom\'as L. G\'omez}

\address{Instituto de Ciencias Matem\'aticas (CSIC-UAM-UC3M-UCM),
Serrano 113bis, 28006 Madrid, Spain; and
Facultad de Ciencias Matem\'aticas,
Universidad Complutense de Madrid, 28040 Madrid, Spain}

\email{tomas.gomez@icmat.es}

\author[N. Hoffmann]{Norbert Hoffmann}

\address{Mathematisches Institut der Freien Universit\"at, Arnimallee 
3, 14195 Berlin, Germany}

\email{norbert.hoffmann@fu-berlin.de}

\author[A. Hogadi]{Amit Hogadi}

\address{School of Mathematics, Tata Institute of Fundamental
Research, Homi Bhabha Road, Bombay 400005, India}

\email{amit@math.tifr.res.in}


\date{}

\begin{abstract}
Let $X$ be a smooth projective variety over $\mathbb C$.
We prove that a twisted Higgs vector bundle $(\calE\, ,\theta)$
on $X$ admits an Einstein--Hermitian connection if and
only if $(\calE\, ,\theta)$ is polystable. A similar result for
twisted vector bundles (no Higgs fields) was proved in \cite{Wa}.
Our approach is simpler.

\textsc{R\'esum\'e.} \textbf{Connexions d'Einstein--Hermite sur les 
fibr\'es de Higgs tordus.}
Soit $X$ une vari\'et\'e projective lisse sur $\mathbb C$. Nous 
d\'emontrons
qu'un fibr\'e de Higgs tordu $(\calE\, ,\theta)$ sur $X$ poss\`ede une 
connexion d'Einstein--Hermite si et seulement si $(\calE\, ,\theta)$
est polystable. Un
r\'esultat analogue pour les fibr\'es vectoriels (d\'epourvus d'un
champ de Higgs) a \'et\'e d\'emontr\'e dans \cite{Wa}. Notre approche 
est plus simple.
\end{abstract}

\maketitle

\section{Introduction}\label{intro.}

Donaldson and Uhlenbeck--Yau proved that a vector bundle on a complex 
projective manifold
admits an Einstein--Hermitian connection if and only if it is polystable
\cite{Do}, \cite{UY}. A generalization of
Einstein--Hermitian connections for Higgs bundles was formulated by
Hitchin (for curves) and Simpson (higher dimensions). They proved that
a Higgs bundle $(\calE\, ,\theta)$ admits an Einstein--Hermitian 
connection if and only
if it is polystable \cite{Hi}, \cite{Si}.

Our aim here is to establish a similar result for twisted sheaves on a smooth
complex projective variety. Let $X$ be an irreducible smooth projective variety
over $\mathbb C$. A \textit{twisted vector bundle} on $X$ is a pair
$({\mathcal X}\, ,\calE)$, where
$$
{\mathcal X}\, \longrightarrow\, X
$$
is a gerbe banded by $\mu_n$ (the $n$--th roots of unity) for some $n$, and
$\calE$ is a vector bundle over $\mathcal X$;
see \cite{Li}, \cite{Hu}, \cite{HS}, \cite{Yo} for twisted bundles.
A twisted Higgs bundle on $X$ is a
twisted vector bundle together with a Higgs field on it.

We prove that a twisted Higgs bundle on $X$ admits an 
Einstein--Hermitian connection if and only if it is polystable (see
Theorem \ref{thm1}).

Let $G$ be a connected reductive linear algebraic group defined over 
$\mathbb C$. Theorem \ref{thm1} generalizes to twisted Higgs principal
$G$--bundles (this is explained at the end).

In \cite{Wa}, Wang proved a similar result for twisted vector bundles
without Higgs structure.

\section{Twisted Higgs bundles}

The base field will be $\mathbb C$.
For any positive integer $n$, by $\mu_n$ we will denote the finite 
subgroup of
${\mathbb C}^*$ consisting of the $n$--th roots of $1$.

Let $X$ be an irreducible smooth projective variety. Let
\begin{equation}\label{b1}
f\, :\, {\mathcal X}\,\longrightarrow\, X
\end{equation}
be a gerbe banded by $\mu_n$.
The cotangent bundle of $\mathcal X$ will be denoted by 
$\Omega^1_{\mathcal X}$.
For any nonnegative integer $i$, let $\Omega^i_{\mathcal X}\,:=\, 
\bigwedge^i\Omega^1_{\mathcal X}$ be the $i$--th exterior power.

Let
$$
\calE\,\longrightarrow\, {\mathcal X}
$$
be a vector bundle. Let $End(\calE)\,:=\, \calE\otimes \calE^*$ be the 
endomorphism
bundle. The associative algebra structure of $End(\calE)$ and the 
exterior algebra
structure of $\bigoplus_{i\geq 0}\Omega^i_{\mathcal X}$ together define 
an algebra
structure on $End(\calE)\otimes (\bigoplus_{i\geq 0}\Omega^i_{\mathcal X})$.

A \textit{Higgs field} on $\calE$ is a section $\theta$ of 
$End(\calE)\otimes \Omega^1_{\mathcal X}$
such that the section $\theta\wedge\theta$ of $End(\calE)\otimes 
\Omega^2_{\mathcal X}$ vanishes
identically.

A \textit{Higgs bundle} on $\mathcal X$ is a pair $(\calE\, ,\theta)$, 
where $\calE$ is a
vector bundle on $\mathcal X$, and $\theta$ is a Higgs field on 
$\calE$. A Higgs bundle
on $\mathcal X$ will be called a \textit{twisted Higgs bundle} on $X$. 
Given a Higgs
bundle $(\calE\, ,\theta)$ on $\mathcal X$, a coherent subsheaf $\calF$
of $\calE$ will be called a \textit{Higgs subsheaf} if $\theta(\calF)\, 
\subset\,
\calF\otimes \Omega^1_{\mathcal X}$.

Let $G$ be a complex linear algebraic group. A \textit{Higgs 
$G$--bundle} on $X$ is a
principal $G$--bundle $E_G\, \longrightarrow\, X$ and a section
$\beta\, \in\, H^0(X, \,\text{ad}(E_G)\otimes\Omega^1_X)$ such that 
$\beta\wedge\beta\,=\,0$,
where $\text{ad}(E_G)$ is the adjoint vector bundle.

Fix a very ample line bundle $L$ over $X$. The \textit{degree} of a 
torsionfree coherent
sheaf $\calF$ on $\mathcal X$ will be defined to be 
$\text{degree}((\det \calF)^{\otimes n})/n^2
\,\in\, {\mathbb Q}$. Note that $(\det \calF)^{\otimes n}$ descends to 
a line bundle on $X$;
its degree is computed using $L$. Fix a K\"ahler form $\omega_X$ on $X$ 
representing
$c_1(L)$. Since the morphism $f$ in \eqref{b1} in \'etale, the pullback
\begin{equation}\label{wom}
\omega_{\mathcal X}\, :=\, f^*\omega_X
\end{equation}
is a K\"ahler form on $\mathcal X$. A Higgs bundle $(\calE\, ,\theta)$ 
is called
\textit{stable} (respectively, \textit{semistable}) if for every Higgs 
subsheaf
$\calF$ with $1\, \leq\, {\rm rank}(\calF)\, < \text{rank}(\calE)$, the 
inequality
$$
\frac{\text{degree}(\calF)}{\text{rank}(\calF)}\, <\, 
\frac{\text{degree}(\calE)}{\text{rank}(\calE)}
~\,~{\rm (respectively,}~\frac{\text{degree}(\calF)}{\text{rank}(\calF)}
\, \leq\,
\frac{\text{degree}(\calE)}{\text{rank}(\calE)}{\rm )}
$$
holds. A semistable Higgs bundle is called \textit{polystable} if it is 
a direct sum of stable Higgs bundles.

For any vector bundle $\calE\,\longrightarrow\, {\mathcal X}$, we have a 
decomposition $\calE\,=\,\bigoplus_{\chi\in\mu^*_n}\calE_\chi$. Henceforth, we
will consider vector bundles $\calE$ with $\calE_\chi\,\not=\, 0$ for at
most one character $\chi$.

Define the homomorphism
\begin{equation}\label{b2}
\rho\,:\, \text{GL}(r, {\mathbb C}) \,\longrightarrow\, 
\text{PGL}(r,{\mathbb C}) \times {\mathbb G}_m \,=:\, H
\end{equation}
by sending $A$ to the class of $A$ and to $(\det A)^n$. Given a vector 
bundle $\calE\,\longrightarrow\, {\mathcal X}$,
the extension of its structure group along $\rho$ defines a principal 
$H$-bundle $\calE_H\,\longrightarrow\, {\mathcal X}$.
Since the inertia $\mu_n$ acts trivially on $\calE_H$, it descends to a 
principal $H$--bundle $E_H\,\longrightarrow\, X$.
A Higgs field $\theta$ on $\calE$ induces a Higgs field $\theta_H$ on 
$\calE_H$. This Higgs field $\theta_H$ on $\calE_H$ 
descends to a Higgs field on $E_H$, which we again denote by $\theta_H$.

The definitions of Higgs (semi)stable and polystable principal bundles
are recalled in \cite[p. 551]{BS}, \cite{DP}.

\begin{lemma}\label{lem-b1}
A Higgs bundle $(\calE\, ,\theta)$ on ${\mathcal X}$ is polystable if and only
if the induced Higgs $H$--bundle $( E_H, \theta_H)$ on $X$ is polystable.
\end{lemma} 

\begin{proof}
The central isogeny $\rho$ in \eqref{b2} produces a bijection of 
parabolic subgroups.
For any parabolic subgroup $P\, \subset\, \text{GL}(r, {\mathbb C})$, 
there is a natural bijective
correspondence between the reductions of structure group of the 
principal $\text{GL}(r, {\mathbb C})$--bundle $\calE$
to $P$ over any open subset $f^{-1}(U)$ and the reductions of structure 
group of the principal $H$--bundle
$E_H$ to $\rho(P)$ over $U$. This bijection proves the lemma.
\end{proof}

\section{Einstein--Hermitian connection on polystable twisted Higgs bundles}

A \textit{Hermitian structure} on a vector bundle
$\calE$ on $\mathcal X$ is a smooth inner product on the fibers
which is invariant under the action of $\mu_n$ on 
the fibers of $\calE$. A Hermitian
structure on $\calE$ produces a $C^\infty$ complex connection on $\calE$. Let
$(\calE\, ,\theta)$ be a Higgs bundle. An \textit{Einstein--Hermitian
connection} on $(\calE\, ,\theta)$ is a Hermitian structure on $\calE$ such
that corresponding connection $\nabla$ on $\calE$ has the following
property:
$$
\Lambda_{\omega_{\mathcal X}} (\text{Curv}(\nabla) +[\theta \, 
,\theta^*]) \,=\, c\cdot
\text{Id}_{\calE}\, ,
$$
for some constant scalar $c$, where $\Lambda_{\omega_{\mathcal X}}$
is the adjoint of multiplication
by the K\"ahler form $\omega_{\mathcal X}$ (see \eqref{wom}),
$\text{Curv}(\nabla)$ is the curvature of $\nabla$,
and $\theta^*$ is the adjoint of $\theta$ constructed using the Hermitian
form on $\calE$.
 
\begin{theorem}\label{thm1}
Let $(\calE\, ,\theta)$ be a twisted Higgs bundle on $X$. Then
$(\calE\, ,\theta)$ is polystable if and only if it admits an
Einstein--Hermitian connection.
\end{theorem}

\begin{proof}
Let $(\calE\, ,\theta)$ be a Higgs bundle on $\mathcal X$. First
assume that $(\calE\, ,\theta)$ is polystable. From Lemma \ref{lem-b1} 
we know that the
induced Higgs $H$--bundle $(E_H\, ,\theta_H)$ on $X$ is polystable. A 
polystable
Higgs $H$--bundle on $X$ admits an Einstein--Hermitian connection 
\cite{Si}, \cite{BS}.
Since $(E_H\, , \theta_H)$ is the descent of $(\calE_H\, , \theta_H)$, an
Einstein--Hermitian connection on $(E_H\, , \theta_H)$ produces
an Einstein--Hermitian connection on $(\calE_H\, , \theta_H)$. A
connection on $\calE_H$ defines
connection on $\calE$ because the homomorphism of Lie algebras
$$
\text{Lie}(\text{GL}(r, {\mathbb C}))\, \longrightarrow\, \text{Lie}(H)
$$
induced by the homomorphism $\rho$ in \eqref{b2} is an isomorphism.
The connection on $(\calE\, ,\theta)$ induced by an
Einstein--Hermitian connection on $(\calE_H\, , \theta_H)$ is
clearly Einstein--Hermitian.

Conversely, an Einstein--Hermitian connection on $(\calE\, ,\theta)$ induces
an Einstein--Hermitian connection on the associated Higgs $H$--bundle 
$(\calE_H\, , \theta_H)$,
which, in turn, induces an Einstein--Hermitian connection on the descended
Higgs $H$--bundle $(E_H\, , \theta_H)$. Therefore, the Higgs $H$--bundle
$(E_H\, , \theta_H)$ is polystable. Hence from Lemma \ref{lem-b1} we 
conclude that
the Higgs bundle $(\calE\, ,\theta)$ is polystable.
\end{proof}

Let $G$ be a connected reductive linear algebraic group defined over
$\mathbb C$. Let $Z$ be the center of $G$; define $G'\,:=\, [G,G]$.

The above theorem holds for principal Higgs $G$--bundles on
$\mathcal{X}$. The proof is the same, but, instead of the
homomorphism \eqref{b2}, we use the homomorphism
$$
\rho: G \longrightarrow H:= G/Z \times (G/G'): g\mapsto (p(g),q(g)^n) 
\; ,
$$
where $p:G\longrightarrow G/Z$ and $q:G\longrightarrow G/G'$ are the
natural projections. Note that $G/G'\,\cong\,
\mathbb{C}^*\times \cdots \times \mathbb{C}^*$.

\medskip
\noindent
\textbf{Acknowledgements.}\, We thank the referee for comments.
The first author thanks Freie Universit\"at Berlin for hospitality
while the work was carried out; the visit was
supported by the SFB 647: Raum - Zeit - Materie.
The second author was supported in part by grant
MTM2007-63582 of the Spanish Ministerio de Educaci\'on y
Ciencia. The third author was supported by the SFB 647:
Raum - Zeit - Materie.


\end{document}